\documentclass[12pt]{amsart}
\usepackage{amssymb,amsmath,amsthm}
\usepackage{tikz}

\usetikzlibrary{arrows,shapes,automata,backgrounds,decorations,petri,positioning,patterns}

\theoremstyle{plain}
\newtheorem{thm}{Theorem}[section]

\newtheorem{cor}[thm]{Corollary}

\newtheorem{prop}[thm]{Proposition}
\theoremstyle{definition}
\newtheorem{defn}[thm]{Definition}
\newtheorem{ex}[thm]{Example}
\newtheorem*{repp@ex}{\repp@title (continued)}
\newcommand{\newreppex}[2]{
\newenvironment{repp#1}[1]{
 \def\repp@title{#2 \ref{##1}}
 \begin{repp@ex}}
 {\end{repp@ex}}}
\makeatother
\newreppex{ex}{Example}
\theoremstyle{remark}
\newtheorem{remark}[thm]{Remark}

\newcommand{\mf}[1]{\mbox{$\mathfrak #1$}}

\newcommand{\poi}[1]{\mbox{$\Lambda_{#1}$}}

\title{Intervals and factors in the Bruhat order}

\author{Bridget Eileen Tenner}
\address{Department of Mathematical Sciences, DePaul University, Chicago, IL 60614}
\email{bridget@math.depaul.edu}
\thanks{Research partially supported by a DePaul University Competitive Research Leave grant.}


\begin{document}

\begin{abstract}
In this paper we study those generic intervals in the Bruhat order of the symmetric group that are isomorphic to the principal order ideal of a permutation $w$, and consider when the minimum and maximum elements of those intervals are related by a certain property of their reduced words.  We show that the property does not hold when $w$ is a decomposable permutation, and that the property always holds when $w$ is the longest permutation.\\

\end{abstract}

\maketitle

The interval structure of the Bruhat order on the symmetric group is not well understood.  One reason for this is that, sometimes, even those intervals that are isomorphic to principal order ideals are actually making use of the fact that a reduced word for the minimum element in the interval can be formed by deleting an \emph{arbitrary} subword of symbols from a reduced word of the maximum element in the interval.  The purpose of this paper is to gain a better understanding of that phenomenon.  More precisely, we explore the principal order ideals $\poi{w}$ with the property that whenever $[x,y]$ is isomorphic to $\poi{w}$, one may obtain a reduced word for $x$ by deleting a consecutive subword from a reduced word for $y$.  Note that deletion of some subword will always produce a reduced word for $x$, but not necessarily the deletion of a consecutive one.  The possibility for a consecutive such word is what we highlight in this work, and what we refer to as ``forcing'' a factor, as defined in Definition~\ref{defn:forcing a factor}. Structural analyses of intervals and principal order ideals are of particular interest because it follows from \cite{brenti} that the Kazhdan-Lusztig polynomial corresponding to the interval $[x,y]$ depends only on the principal order ideal $\poi{y}$.

The precise question we answer here is laid out in Section~\ref{section:intro}, along with the relevant objects and examples.  Section~\ref{section:definitions} gives additional definitions, and the main results appear in Theorems~\ref{thm:doesn't force} and~\ref{thm:forces}.  Finally, Section~\ref{section:questions} discusses directions for subsequent work.

\section{Introduction}\label{section:intro}

The symmetric group $\mf{S}_n$ on $\{1,\ldots,n\}$ is generated by $\{s_1,\ldots,s_{n-1}\}$, where $s_i$ is the permutation interchanging $i$ and $i+1$, and fixing all other elements.  These generators, known as \emph{simple reflections}, satisfy the Coxeter relations
$$\begin{array}{rcll}
s_i^2 &=& 1 &\text{\ \ for all } i,\\
s_is_j &=& s_js_i &\text{\ \ if } |i-j| > 1, \text{and}\\
s_is_{i+1}s_i &=& s_{i+1}s_is_{i+1} &\text{\ \ for } 1 \le i \le n-2.
\end{array}$$

A permutation $w$ can be recorded as a product of simple reflections
$$w = s_{i_1}s_{i_2} \cdots s_{i_r},$$
of which there are infinitely many representations, or written uniquely in one-line notation
$$w = w(1)w(2) \cdots w(n).$$
These two representations of the same class of objects are quite different, and each have advantages in certain contexts.  The main result of \cite{rdpp} was a way to translate between the two options.

The order in which we compose maps indicates that $s_iw$ interchanges the positions of the values $i$ and $i+1$ in the one-line notation of $w$, and $ws_i$ interchanges the values in positions $i$ and $i+1$ in the one-line notation of $w$.

\begin{ex}\label{ex:notation}
The permutation $w \in \mf{S}_4$ which maps $1$ to $3$, $2$ to itself, $3$ to $4$, and $4$ to $1$, would be written in one-line notation as
$$w = 3241.$$
A few of the infinite many ways to express $w$ as a product of simple reflections include:
\begin{eqnarray*}
w &=& s_1s_2s_1s_3\\
&=& s_2s_1s_2s_3\\
&=& s_1s_2s_3s_1\\
&=& s_1s_3s_3s_2s_3s_1.
\end{eqnarray*}
\end{ex}

As demonstrated in Example~\ref{ex:notation}, the number of simple reflections involved in representing a particular permutation is not fixed.  There is, however, a minimum value.

\begin{defn}\label{defn:decomp word length}
If $w = s_{i_1}\cdots s_{i_{\ell(w)}}$ where $\ell(w)$ is minimal, then $s_{i_1}\cdots s_{i_{\ell(w)}}$ is a \emph{reduced decomposition} of $w$, and the string of subscripts ${\sf i_1\cdots i_{\ell(w)}}$ is a \emph{reduced word} of $w$.  The set of reduced words of $w$ is denoted $R(w)$.  The nonnegative integer $\ell(w)$ is the \emph{length} of $w$.
\end{defn}

To avoid confusion with permutations, which are also strings of integers, reduced words and their symbols will be written in sans serif.

The Coxeter relations among the symbols in a reduced decomposition have obvious analogues for the symbols in a reduced word:
$$\begin{array}{rclcrcll}
s_is_j &=& s_js_i &\ \longleftrightarrow\ & {\sf ij} &=& {\sf ji} &\text{\ \ if } |i-j| > 1, \text{and}\\
s_is_{i+1}s_i &=& s_{i+1}s_is_{i+1} &\ \longleftrightarrow\ & {\sf i(i+1)i} &=& {\sf (i+1)i(i+1)} &\text{\ \ for } 1 \le i \le n-2.
\end{array}$$
A product of simple reflections is \emph{reduced} if it is a reduced decomposition for some permutation, and a string of integers is \emph{reduced} if the corresponding product of simple reflections is reduced.

\begin{ex}
Continuing Example~\ref{ex:notation}, $\ell(3241) = 4$ and $R(3241) = \{{\sf 1213}, {\sf 2123}, {\sf 1231}\}$.  (Note that we have not explained how to compute these values.)
\end{ex}

The Bruhat order is a partial ordering given to the elements of a Coxeter group.  Although we are only concerned with the symmetric group, we give the full definition here.

\begin{defn}\label{defn:bruhat}
Suppose that $x$ and $y$ are elements of a Coxeter group.  Then $x \le y$ in the Bruhat order if there exists a reduced word of $x$ that can be obtained from a reduced word of $y$ by deleting symbols and simplifying Coxeter relations as necessary.
\end{defn}

The Bruhat order gives a graded poset structure to any Coxeter group, where the rank of an element is the element's length.  The Bruhat order of $\mf{S}_4$ is depicted in Figure~\ref{fig:s4}.  Of course, Definition~\ref{defn:bruhat} can be stated analogously in terms of reduced decompositions.

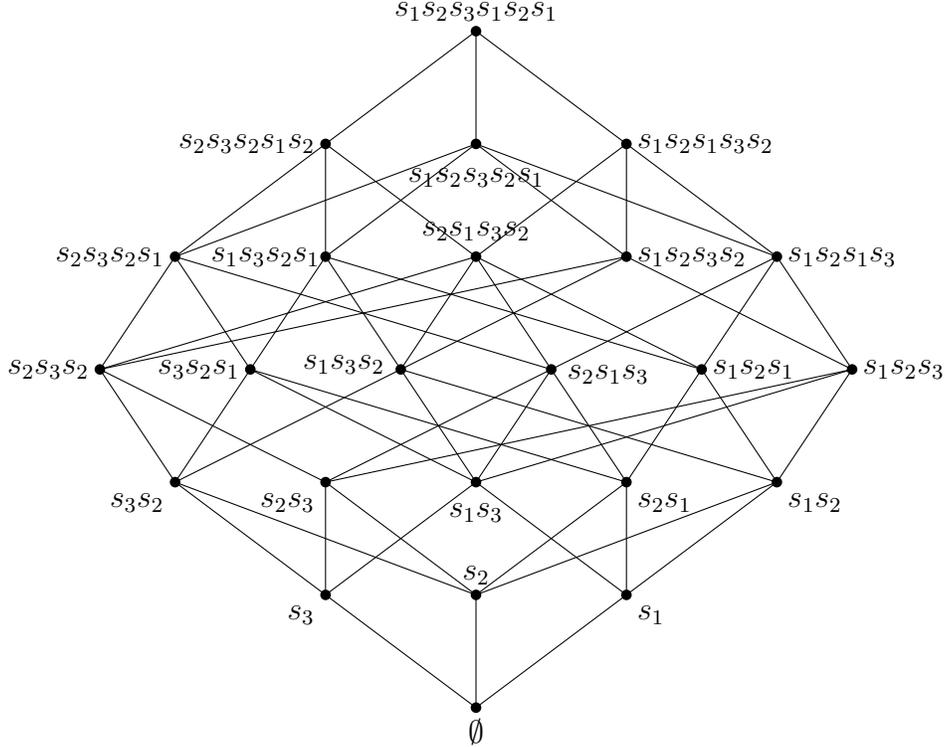
\begin{figure}[htbp]
\begin{tikzpicture}[scale=1]
\foreach \y in {0,9} {\fill[black] (0,\y) circle (2pt);}
\foreach \x in {-2,0,2} {\foreach \y in {1.5,7.5} {\fill[black] (\x,\y) circle (2pt);};}
\foreach \x in {-4,-2,0,2,4} {\foreach \y in {3,6} {\fill[black] (\x,\y) circle (2pt);};}
\foreach \x in {-5,-3,-1,1,3,5} {\fill[black] (\x,4.5) circle (2pt);}
\draw (0,0) coordinate (1234); \draw (0,9) coordinate (4321);
\draw (-2,1.5) coordinate (1243); \draw (0,1.5) coordinate (1324); \draw (2,1.5) coordinate (2134);
\draw (-2,7.5) coordinate (4312); \draw (0,7.5) coordinate (4231); \draw (2,7.5) coordinate (3421);
\draw (-4,3) coordinate (1423); \draw (-2,3) coordinate (1342); \draw (0,3) coordinate (2143); \draw (2,3) coordinate (3124); \draw (4,3) coordinate (2314);
\draw (-4,6) coordinate (4132); \draw (-2,6) coordinate (4213); \draw (0,6) coordinate (3412); \draw (2,6) coordinate (2431); \draw (4,6) coordinate (3241);
\draw (-5,4.5) coordinate (1432); \draw (-3,4.5) coordinate (4123); \draw (-1,4.5) coordinate (2413); \draw (1,4.5) coordinate (3142); \draw (3,4.5) coordinate (3214); \draw (5,4.5) coordinate (2341);
\draw (1234) -- (1243) -- (1423) -- (1432) -- (4132) -- (4312) -- (4321) -- (3421) -- (3241) -- (2341) -- (2314) -- (2134) -- (1234) -- (1324) -- (1423) -- (4123) -- (4132) -- (4231) -- (4321);
\draw (1243) -- (1342) -- (1432) -- (3412) -- (4312);
\draw (1243) -- (2143) -- (4123) -- (4213) -- (4312);
\draw (1324) -- (1342) -- (3142) -- (4132);
\draw (1324) -- (3124) -- (4123);
\draw (1324) -- (2314) -- (2413) -- (4213) -- (4231);
\draw (2134) -- (2143) -- (2413) -- (3412) -- (3421);
\draw (2134) -- (3124) -- (3142) -- (3412);
\draw (1423) -- (2413) -- (2431) -- (4231);
\draw (1342) -- (2341) -- (2431) -- (3421);
\draw (2143) -- (3142) -- (3241) -- (4231);
\draw (2143) -- (2341);
\draw (3124) -- (3214) -- (4213);
\draw (2314) -- (3214) -- (3412);
\draw (1432) -- (2431);
\draw (3214) -- (3241);
\draw (1234) node[below] {$\emptyset$};
\draw (4321) node[above] {$s_1s_2s_3s_1s_2s_1$};
\draw (2134) node[below right] {$s_1$}; \draw (1324) node[above] {$s_2$}; \draw (1243) node[below left] {$s_3$};
\draw (1342) node[below left] {$s_2s_3$}; \draw (1423) node[below left] {$s_3s_2$}; \draw (2143) node[below, yshift=-4pt] {$s_1s_3$}; \draw (2314) node[below right] {$s_1s_2$}; \draw (3124) node[below right] {$s_2s_1$};
\draw (1432) node[left] {$s_2s_3s_2$}; \draw (3214) node[right] {$s_1s_2s_1$}; \draw (4123) node[left] {$s_3s_2s_1$}; \draw (2341) node[right] {$s_1s_2s_3$}; \draw (2413) node[left, xshift=-2pt, yshift=2pt] {$s_1s_3s_2$}; \draw (3142) node[right, xshift=2pt, yshift=-2pt] {$s_2s_1s_3$};
\draw (4132) node[left] {$s_2s_3s_2s_1$}; \draw (4213) node[left, xshift=2pt] {$s_1s_3s_2s_1$}; \draw (3241) node[right] {$s_1s_2s_1s_3$}; \draw (2431) node[right] {$s_1s_2s_3s_2$}; \draw (3412) node[above, yshift=2pt] {$s_2s_1s_3s_2$};
\draw (4231) node[below, yshift=-5pt] {$s_1s_2s_3s_2s_1$}; \draw (4312) node[left] {$s_2s_3s_2s_1s_2$}; \draw (3421) node[right] {$s_1s_2s_1s_3s_2$};
\end{tikzpicture}
\caption{The Bruhat order of the symmetric group $\mf{S}_4$, where each element is labeled by one of its (possibly many) reduced decompositions.}\label{fig:s4}
\end{figure}

The Bruhat order gives a partial ordering to the permutations of a given set.  Another way to describe it is as a subword order on reduced decompositions/words.  Dyer studied intervals in the Bruhat order for all finite Coxeter groups, showing that there are only finitely many non-isomorphic intervals of any given length \cite{dyer}.  Jantzen and Hultman have classified all possible length $4$ intervals, as well as the length $5$ intervals in the symmetric group \cite{hultman1, hultman2, jantzen}.  The special class of intervals for which the minimal element is the identity, known as principal order ideals and defined formally in Definition~\ref{defn:poi}, were studied by the author in \cite{patt-bru}.  As is obvious from Table~\ref{table:counting intervals and pois}, there are intervals in $\mf{S}_n$ that never appear as principal order ideals.

\begin{table}[htbp]
\centering
\begin{tabular}{r|cccccc}
Length & $0$ & $1$ & $2$ & $3$ & $4$ & $5$\\
\hline
\# Non-isomorphic intervals in $\mf{S}_n$ & $1$ & $1$ & $1$ & $3$ & $7$ & $25$\\
\# Non-isomorphic principal order ideals in $\mf{S}_n$ & $1$ & $1$ & $1$ & $2$ & $3$ & $5$\\
\end{tabular}
\vspace{.1in}
\caption{The number of non-isomorphic intervals and principals order ideals of length at most $5$ (as defined in Definition~\ref{defn:decomp word length}) appearing in the Bruhat order of symmetric groups.}\label{table:counting intervals and pois}
\end{table}

\begin{ex}\label{ex:interval not poi}
The interval $[2143, 4231]$ in the Bruhat order of $\mf{S}_4$, depicted in Figure~\ref{fig:2143,4231}, never appears as a principal order ideal in the Bruhat order of any symmetric group.
\end{ex}

\begin{figure}[htbp]
\begin{tikzpicture}
\foreach \y in {0,4.5} {\fill[black] (0,\y) circle (2pt);}
\foreach \y in {1.5,3} {\foreach \x in {-3,-1,1,3} {\fill[black] (\x,\y) circle (2pt);};}
\draw (0,0) coordinate (2143);
\draw (-3,1.5) coordinate (2341); \draw (-1,1.5) coordinate (2413); \draw (1,1.5) coordinate (3142); \draw (3,1.5) coordinate (4123);
\draw (-3,3) coordinate (2431); \draw (-1,3) coordinate (3241); \draw (1,3) coordinate (4213); \draw (3,3) coordinate (4132);
\draw (0,4.5) coordinate (4231);
\draw (2143) node[below] {$2143$};
\draw (4231) node[above] {$4231$};
\foreach \x in {2341,2431,2413,3241} {\draw (\x) node[left] {$\x$};}
\foreach \x in {4123,4132,3142,4213} {\draw (\x) node[right] {$\x$};}
\foreach \x in {2341,2413,3142,4123} {\draw (2143) -- (\x);}
\foreach \x in {2431,3241,4213,4132} {\draw (4231) -- (\x);}
\draw (2341) -- (2431) -- (2413) -- (4213) -- (4123) -- (4132) -- (3142) -- (3241) -- (2341);
\end{tikzpicture}
\caption{The interval $[2143,4231]\subset \mf{S}_4$, not isomorphic to any principal order ideal appearing in the Bruhat order of any $\mf{S}_n$.}\label{fig:2143,4231}
\end{figure}
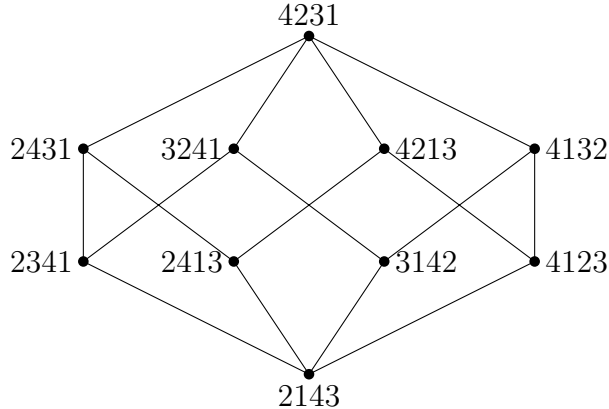

The purpose of this paper is to explore how and when generic intervals in $\mf{S}_n$ might be isomorphic to principal order ideals.  So that we can refer to it subsequently, we highlight the following discussion in a remark.

\begin{remark}\label{remark}
Principal order ideals in $\mf{S}_n$ are special cases of intervals, and they lack a certain freedom that more general intervals possess.  In the Bruhat order, $x \le y$ if there is an element of $R(x)$ that is a subword of an element of $R(y)$.  In a principal order ideal, this $x$ is the identity permutation, and so $R(x) = \{\emptyset\}$.  Of course $\emptyset$ is a subword of every word.  Thus, for principal order ideals, the entire reduced word for $y$ must be deleted in order to yield the reduced word for $x$.  In particular, note that what is getting deleted is a consecutive subword (in fact, the entire word itself).  On the other hand, in a generic interval, what gets deleted from the reduced word for $y$ need not be a consecutive subword.  This potential yields some generic intervals that do not appear as principal order ideals.
\end{remark}

\begin{ex}\label{ex:21543 52341}
The interval $[21543,52341]$ is isomorphic to the interval in Figure~\ref{fig:2143,4231}.  Note that
$$R(21543) = \{{\sf 1343},{\sf 3143}, {\sf 3413}, {\sf 3431}, {\sf 1434}, {\sf 4134}, {\sf 4314}, {\sf 4341}\}$$
and
\begin{align*}
R(52341) =& \ \{ {\sf 1234321}, {\sf 1243421}, {\sf 1423421}, {\sf 4123421}, {\sf 1243241}, {\sf 1423241}, {\sf 4123241}\\
&\phantom{\ \{}{\sf 1243214}, {\sf 1423214}, {\sf 4123214}, {\sf 1432341}, {\sf 4132341}, {\sf 4312341}, {\sf 1432314},\\
&\phantom{\ \{}{\sf 4132314}, {\sf 4312314}, {\sf 1432134}, {\sf 4132134}, {\sf 4312134}, {\sf 4321234}\}.
\end{align*}
Thus no element of $R(21543)$ can be obtained by deleting a consecutive subword of symbols from an element of $R(52341)$.
\end{ex}

The purpose of the current work is to examine when a principal order ideal in the symmetric group also appears as a more general interval.  In particular, we want to understand when such an interval, not necessarily beginning at rank $0$ in the poset, must still come from deleting a factor from a reduced word of its maximum element.

Before making this property precise, consider the following result, where vexillary permutations are exactly those that avoid the pattern $2143$.

In \cite{rdpp}, we showed that if a permutation $w$ contains a vexillary $p$-pattern, then there is a reduced word for $w$ possessing a reduced word {\sf f} for $p$ as a factor, possibly with a fixed positive integer added to each letter in {\sf f}.  If $v$ is the permutation obtained by deleting this subword from $w$, then this implies that the interval $[v,w]$ in the Bruhat order would be isomorphic to $\poi{p}$, the principal order ideal for $p$.  Indeed, this would also imply that there is a copy of $\poi{p}$ sitting as a principal order ideal inside of the principal order ideal $\poi{w}$ of $w$, as depicted in Figure~\ref{fig:vexillary result}.

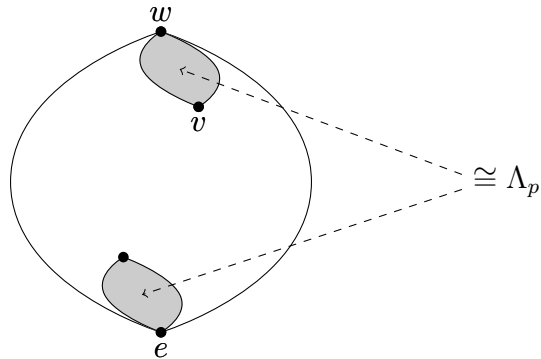
\begin{figure}[htbp]
\begin{tikzpicture}
\fill[black] (0,0) circle (2pt) node[below] {$e$};
\fill[black] (0,4) circle (2pt) node[above] {$w$};
\fill[black] (.5,3) circle (2pt) node[below] {$v$};
\draw plot [smooth,tension=1.25] coordinates{(0,0) (-2,2) (0,4)};
\draw plot [smooth,tension=1.25] coordinates{(0,0) (2,2) (0,4)};
\filldraw[black!20] plot [smooth,tension=1] coordinates {(.5,3) (-.25,3.5) (0,4)};
\filldraw[black!20] plot [smooth,tension=1] coordinates {(.5,3) (.75,3.5) (0,4)};
\draw plot [smooth,tension=1] coordinates {(.5,3) (-.25,3.5) (0,4)};
\draw plot [smooth,tension=1] coordinates {(.5,3) (.75,3.5) (0,4)};
\filldraw[black!20] plot [smooth,tension=1] coordinates {(0,0) (-.75,.5) (-.5,1)};
\filldraw[black!20] plot [smooth,tension=1] coordinates {(0,0) (.25,.5) (-.5,1)};
\draw plot [smooth,tension=1] coordinates {(0,0) (-.75,.5) (-.5,1)};
\draw plot [smooth,tension=1] coordinates {(0,0) (.25,.5) (-.5,1)};
\filldraw[black!20] (0,0) -- (-.5,1);
\filldraw[black!20] (.5,3) -- (0,4);
\fill[black] (0,0) circle (2pt) node[below] {$e$};
\fill[black] (0,4) circle (2pt) node[above] {$w$};
\fill[black] (.5,3) circle (2pt) node[below] {$v$};
\fill[black] (-.5,1) circle (2pt);
\draw[->,dashed] (4,2.1) -- (.25,3.5);
\draw[->,dashed] (4,1.9) -- (-.25,.5);
\draw (4,2) node[right] {$\cong \poi{p}$};
\draw (-4,2) node[left] {\phantom{$\cong \poi{p}$}};
\end{tikzpicture}
\caption{When a permutation $w$ contains a vexillary pattern $p$, then $\poi{w}$ will contain intervals isomorphic to $\poi{p}$ as indicated by the shading.  (There may, of course, be other intervals in $\poi{w}$ isomorphic to $\poi{p}$ as well.)}\label{fig:vexillary result}
\end{figure}

\begin{ex}
The permutation $4213$ contains the vexillary pattern $321$.  The interval $[1243,4213]$ and the principal order ideal $\poi{3214}$ are both isomorphic to $\poi{321}$, as shown in Figure~\ref{fig:1234,4213}.
\end{ex}

\begin{figure}[htbp]
\begin{tikzpicture}
\draw (0,0) coordinate (1234);
\draw (-2,1.5) coordinate (2134); \draw (0,1.5) coordinate (1243); \draw (2,1.5) coordinate (1324);
\draw (-3,3) coordinate (3124); \draw (-1,3) coordinate (2143); \draw (1,3) coordinate (1423); \draw (3,3) coordinate (2314);
\draw (-2,4.5) coordinate (3214); \draw (0,4.5) coordinate (4123); \draw (2,4.5) coordinate (2413);
\draw (0,6) coordinate (4213);
\foreach \x in {1234} {\draw (\x) node[below] {$\x$};}
\foreach \x in {2134,3124,2143,3214,1243} {\draw (\x) node[left] {$\x$};}
\foreach \x in {1324,1423,2314,2413,4123} {\draw (\x) node[right] {$\x$};}
\foreach \x in {4213} {\draw (\x) node[above] {$\x$};}
\draw (1234) -- (1243); \draw (2134) -- (2143); \draw (1324) -- (1423);
\draw (3124) -- (4123); \draw (2314) -- (2413); \draw (3214) -- (4213);
\draw[ultra thick] (1243) -- (1423) -- (2413) -- (4213) -- (4123) -- (2143) -- (1243);
\draw[ultra thick] (2143) -- (2413); \draw[ultra thick] (1423) -- (4123);
\draw[ultra thick,dashed] (1234) -- (1324) -- (2314) -- (3214) -- (3124) -- (2134) -- (1234);
\draw[ultra thick,dashed] (2134) -- (2314); \draw[ultra thick, dashed] (1324) -- (3124);
\foreach \x in {1234,2134,1324,3124,2314,3214} {\fill[white] (\x) circle (2pt); \draw[very thick] (\x) circle (2pt);}
\foreach \x in {1243,2143,1423,4123,2413,4213} {\fill[black] (\x) circle (2pt); \draw (\x) circle (2pt);}
\end{tikzpicture}
\caption{The principal order ideal $\poi{4213}$ has a copy of the principal order ideal $\poi{321}$, both as a principal order ideal (thick dashed lines) and as a generic interval (thick solid lines).}\label{fig:1234,4213}
\end{figure}
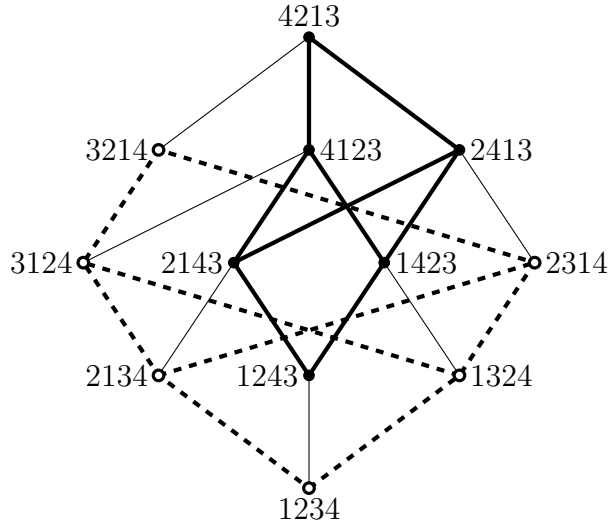

We are now able to clarify the main question of this paper: when does a principal order ideal $\poi{w} \subseteq \mf{S}_n$ have the property that for all intervals $[x,y] \cong \poi{w}$, there is a reduced word ${\sf i} \in R(x)$ and a reduced word ${\sf j} \in R(y)$ such that ${\sf i}$ can be formed by deleting a consecutive subword from ${\sf j}$?

The main results of this paper are that no decomposable permutation can force a factor (Theorem~\ref{thm:doesn't force}), and that the permutation $n(n-1)\cdots 321$ does force a factor for all $n$ (Theorem~\ref{thm:forces}).  These are preceded by a discussion of the useful terminology in Section~\ref{section:definitions}, and the paper concludes with suggestions for future work.

\section{Definitions}\label{section:definitions}

We now define the main objects of the current work.  More details and background information about Coxeter groups can be found in \cite{bjorner-brenti} and \cite{macdonald}.

Two of the most fundamental structural features of a poset are its intervals and its order ideals.  In a poset with a unique minimal element, these objects intersect at the concept of a principal order ideal.

\begin{defn}\label{defn:interval}
Consider a poset $P$ and elements $x,y \in P$ with $x \le y$.  The set of elements $\{z : x \le z \le y\} \subseteq P$ is denoted $[x,y]$, and is called an \emph{interval}.
\end{defn}

\begin{defn}\label{defn:order ideal}
An \emph{order ideal} in a poset $P$ is a subset $I \subseteq P$ such that if $y \in I$ and $x \le y$, then $x \in I$.
\end{defn}

\begin{defn}\label{defn:poi}
A \emph{principal order ideal} in a poset $P$ is an order ideal with a unique maximal element.  Equivalently, the principal order ideals of $P$ are the subsets $\poi{y} = \{z : z \le y\}$ for each $y \in P$.  If $P$ has a unique minimal element $\hat{0}$, then the principal order ideals are exactly the intervals of the form $[\hat{0},y]$.
\end{defn}

To calculate the length of a permutation, we must make a preliminary definition.

\begin{defn}\label{defn:inversion}
Given a permutation $w$, an \emph{inversion} in $w$ is a pair $(i,j)$ such that $i<j$ and $w(i) > w(j)$.
\end{defn}

Inversions are easy to see in the one-line notation of a permutation: they consist of a value (the $w(i)$ of the definition) appearing somewhere to the left of a smaller value (the $w(j)$ of the definition).  It is well known that $\ell(w)$ is equal to the number of inversions in $w$, and it is now clear that there is a unique permutation in $\mf{S}_n$ having maximal length.

\begin{defn}
Fix a positive integer $n$.  The unique element of maximal length in $\mf{S}_n$ is $w_0^n = n(n-1)\cdots 321$, and $\ell(w_0^n) = \binom{n}{2}$.
\end{defn}

As described in Definition~\ref{defn:bruhat}, the relation $x \le y$ in the Bruhat order allows any subset of symbols to be deleted from a reduced word of $y$ in order to form a reduced word of $x$.  In particular, these symbols need not be consecutive.  For example, as shown in Figure~\ref{fig:s4},
$$1324 = s_2 < s_1s_2s_3 = 2341.$$
In this paper, we will examine when all of the activity happening within an interval $[x,y]$ in the Bruhat order of some symmetric group is actually happening within some consecutive substring of the symbols of an element of $R(y)$, all of which must be deleted to form an element of $R(x)$.

\begin{defn}
A \emph{factor} in a word is a consecutive substring.
\end{defn}

This paper is concerned with understanding when intervals $[x,y]$ that are isomorphic to $\poi{w}$ for some $w$ may or may not be formed in ``interesting'' ways.  This is made more precise in the following definition.

\begin{defn}\label{defn:forcing a factor}
Fix an element $w \in \mf{S}_n$, and consider its principal order ideal $\poi{w}$.  If it is true that for every interval $[x,y] \cong \poi{w}$, there exists an element of $R(x)$ formed by deleting a factor from an element of $R(y)$, then $w$ \emph{forces} a factor.  Otherwise $w$ \emph{does not force a factor}.
\end{defn}

The ``interesting'' feature noted above was described in Remark~\ref{remark} in the introduction to this work.  The first thing to note about this topic is that determining which permutations force a factor is an interesting problem.  More precisely, not all permutations do so.

\begin{ex}\label{ex:2314 doesn't force}
Consider $\poi{2314} \subset \mf{S}_4$.  Note that the principal order ideal $\poi{2314}$ is isomorphic to the interval $[1324, 2341]$.
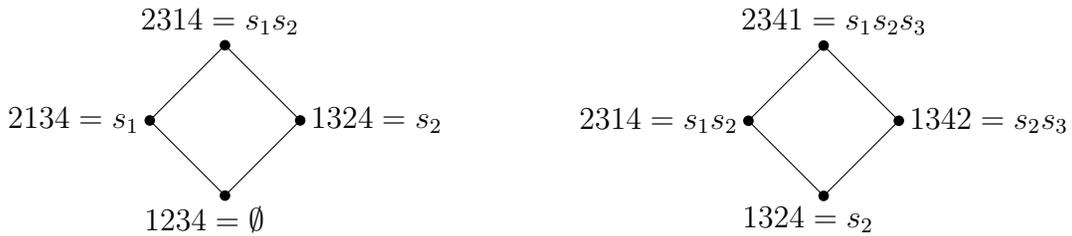
\begin{center}
\begin{tikzpicture}
\fill[black] (0,0) circle (2pt) node[below] {$1234 = \emptyset$\ \ \ \ \ }; \fill[black] (0,2) circle (2pt) node[above] {$2314 = s_1s_2$\,\,};
\fill[black] (-1,1) circle (2pt) node[left] {$2134 = s_1$}; \fill[black] (1,1) circle (2pt) node[right] {$1324 = s_2$};
\draw (0,0) -- (1,1) -- (0,2) -- (-1,1) -- (0,0);
\end{tikzpicture}
\hspace{.5in}
\begin{tikzpicture}
\fill[black] (0,0) circle (2pt) node[below] {$1324 = s_2$\ \ \ \ }; \fill[black] (0,2) circle (2pt) node[above] {\,\, $2341 = s_1s_2s_3$};
\fill[black] (-1,1) circle (2pt) node[left] {$2314 = s_1s_2$}; \fill[black] (1,1) circle (2pt) node[right] {$1342 = s_2s_3$};
\draw (0,0) -- (1,1) -- (0,2) -- (-1,1) -- (0,0);
\end{tikzpicture}
\end{center}
However, there is no element of $R(1324) = \{{\sf 2}\}$ that can be formed from by deleting a single factor from an element of $R(2341) = \{{\sf 123}\}$.  Thus $2314 = s_1s_2$ does not force a factor.
\end{ex}

\section{Permutations that do not force a factor}

In this section we describe a large class of permutations that do not force a factor.

\begin{defn}
A permutation $w$ is \emph{decomposable} if $\poi{w} \cong \poi{u} \times \poi{v}$, where neither $\poi{u}$ nor $\poi{v}$ is itself isomorphic to $\poi{w}$.  If $w$ is not decomposable, then it is \emph{indecomposable}.
\end{defn}

\begin{ex}\label{ex:decomposable}
The permutation $4213 \in \mf{S}_4$ is decomposable because $\poi{4213} \cong \poi{3214} \times \poi{1243}$. This is depicted in Figure~\ref{fig:1234,4213}.
\end{ex}

\begin{prop}[\cite{thesis}]\label{prop:decomp}
A permutation $w \in \mf{S}_n$ is decomposable if and only if there exists $m \in [n-2]$ and a reduced word ${\sf a_1a_2} \in R(w)$, where ${\sf a_1}$ and ${\sf a_2}$ are nonempty, such that ${\sf a_i}$ consists only of letters less than or equal to $m$, and ${\sf a_{3-i}}$ consists only of letters strictly greater than $m$.
\end{prop}

\begin{reppex}{ex:decomposable}
The reduced words of $4213 \in \mf{S}_4$ are $\{{\sf 3121}, {\sf 3212}, {\sf 1321}\}$. Then in the language of Proposition~\ref{prop:decomp}, we can let $m = 2$, and ${\sf a_1} = {\sf 3}$ and ${\sf a_2} = {\sf 121}$. Thus $4213$ is decomposable.
\end{reppex}

We now show, constructively, that no decomposable permutation forces a factor.  This means that for any decomposable permutation $w$, we must produce an interval $[x,y] \cong \poi{w}$ such that no reduced word for $x$ can be obtained by deleting a factor from any reduced word for $y$.

\begin{thm}\label{thm:doesn't force}
If $w$ is decomposable, then $w$ does not force a factor.
\end{thm}

\begin{proof}
Suppose that $w$ is decomposable.  Consider the value $m$ and the reduced word ${\sf a_1a_2} \in R(w)$ guaranteed by Proposition~\ref{prop:decomp}.  We can assume, without loss of generality, that $i=1$ in the proposition.  Let $k_1\le m$ be the largest value appearing in ${\sf a_1}$, and $k_2\ge m+1$ be the smallest value appearing in ${\sf a_{2}}$.

Let ${\sf a_2'}$ be the string obtained from ${\sf a_2}$ by increasing each symbol by $1$.  Also, define ${\sf b} = {\sf (k_1+1) (k_1+2) \cdots (m+1) \cdots (k_2-1)k_2}$, a string of consecutive increasing letters.  Now define $w^- = s_{k_1+1}s_{k_1+2}\cdots s_{m+1} \cdots s_{k_2-1}s_{k_2}$; that is,
$${\sf b} \in R(w^-)$$
It is not hard to see that ${\sf a_1}{\sf b}\,{\sf a_2'}$ is reduced, by construction.  Define $w^+$ so that
$${\sf a_1}{\sf b}\,{\sf a_2'} \in R(w^+).$$
It is not hard to see that
$$\poi{w} \cong [w^-,w^+].$$

Neither ${\sf a_1}$ nor ${\sf a_2'}$ contain any of the letters $\{k_1+1, k_1+2, \ldots, k_2\}$.  Moreover, $k_1 \in {\sf a_1}$ and $k_2+1 \in {\sf a_2'}$.  Thus the Coxeter relations prohibit $k_1 \in {\sf a_1}$ from commuting into or across ${\sf b}$ from the left, and similarly $k_2+1 \in {\sf a_2'}$ cannot do so from the right.  Thus it will be impossible to get $k_1$ and $k_2+1$ into the same factor whose deletion would yield ${\sf b}$.

Therefore $w$ does not force a factor.
\end{proof}

Example~\ref{ex:2314 doesn't force} depicts the procedure outlined in Theorem~\ref{thm:doesn't force}.  In that case, $w = s_1s_2 = 2314$, and so $m=1$, ${\sf a_1} = {\sf 1}$, and ${\sf a_2} = {\sf 2}$.  Then $k_1 = 1$ and $k_2 = 2$, and so ${\sf a_2'} = {\sf 3}$ and the resulting string ${\sf a_1} {\sf b} \,{\sf a_2'} = {\sf 123}$.  Therefore $w^- = s_2 = 1324$ and $w^+ = s_1s_2s_3 = 2341$.  This yields exactly the demonstrative interval $[1324,2341]$ of the example.

Theorem~\ref{thm:doesn't force} says that if a permutation has a principal order ideal that can decompose nontrivially into a direct product of posets, then there are ways for that principal order ideal to appear as an interval in an ``interesting'' way, as described in Remark~\ref{remark}.  In this context, then, the result may not be surprising.  It might even be natural to wonder whether the converse to Theorem~\ref{thm:doesn't force} is also true.  Unfortunately, it is not.

\begin{ex}
Consider the permutation $w = 3412 = s_2s_1s_3s_2 \in \mf{S}_4$.  Because $R(w) = \{{\sf 2132},{\sf 2312}\}$, we see that this $w$ is indecomposable.  It is not hard to check that
$$[12543,52341]\cong\poi{3412}.$$
To show that $w$ does not force a factor, note that $R(12543) = \{{\sf 343},{\sf 434}\}$, and recall $R(52341)$ from Example~\ref{ex:21543 52341}.  There is no element of $R(52341)$ from which a single factor could be deleted to yield either ${\sf 343}$ or ${\sf 434}$.  Thus $w$ does not force a factor.
\end{ex}

\section{Permutations that do force a factor}

In this section we prove that the longest permutation $n(n-1)\cdots 321$ always forces a factor.  With the understanding that the symmetric group is given a poset under the Bruhat order, we will abuse notation slightly and write $\poi{w_0^n} \cong \mf{S}_n$, henceforth.  Thus we now show that for any interval $[x,y]$ appearing in the Bruhat order of a symmetric group satisfying $[x,y] \cong \mf{S}_n$, there exists some ${\sf i} \in R(x)$ that can be obtained from some ${\sf j} \in R(y)$ by deleting a factor.

The main result will be proved inductively, and its proof will benefit from some preliminary results.  The first of these is about generic intervals in the Bruhat order of the symmetric group, not of any fixed isomorphism class.  The proposition concerns the coatoms in an interval $[x,y]$, that is, the elements $w$ of the interval that are covered by $y$ (denoted $x \le w \lessdot y$).

\begin{prop}\label{prop:all positions mentioned in coatoms}
Suppose that $x,y \in \mf{S}_n$ with $x < y$.  Fix $i$ satisfying $x(i) \neq y(i)$.  Then there exists a permutation $w$ with $x \le w \lessdot y$ and $w(i) \neq y(i)$.
\end{prop}

\begin{proof}
We prove the result by induction on $\ell(y) - \ell(x)$.

If $\ell(y) - \ell(x) = 1$, then set $w=x$.  Now consider $\ell(y) - \ell(x) > 1$, and suppose inductively that the result is true for all intervals of length less than $\ell(y) - \ell(x)$.

Fix some $v$ satisfying $x \le v \lessdot y$.  If $v(i) \neq y(i)$, then set $w=v$ and we are done.  If, instead, $v(i) = y(i)$, then we can apply the inductive assumption to the interval $[x,v]$.  This yields a permutation $u$ with $x \le u \lessdot v$ and $u(i) \neq v(i)$.  Consider the interval $[u,y]$.  As described in Table~\ref{table:counting intervals and pois}, this has only one possible form, and includes a fourth element which will denote $w$.
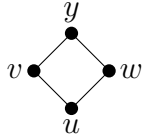
\begin{center}
\begin{tikzpicture}[scale=.5]
\draw (0,0) -- (1,1) -- (0,2) -- (-1,1) -- (0,0);
\foreach \x in {(0,0),(1,1),(0,2),(-1,1)} {\fill[black] \x circle (5pt);}
\draw (0,0) node[below] {$u$};
\draw (0,2) node[above] {$y$};
\draw (-1,1) node[left] {$v$}; 
\draw (1,1) node[right] {$w$};
\end{tikzpicture}
\end{center}
Because $\ell(y) - \ell(u) = 2$, the permutations $u$ and $y$ differ, as strings, in either three or four positions, one of which is necessarily position $i$.

If $u$ and $y$ differ in four positions, then $u$ and $v$ differ in two positions ($i$ and $j$, for some $j$), and $v$ and $y$ differ in two other positions.  The two transpositions commute, and so $w$ is obtained from $y$ by swapping the values in positions $i$ and $j$.  In other words, $w(i) \neq y(i)$.

Suppose, on the other hand, that $u$ and $y$ differ in just three positions: $i$, $j_1$, and $j_2$.  Then, because $v(i) = y(i)$, we have that $v$ and $y$ differ in positions $j_1$ and $j_2$.  Because $w \neq v$, the two positions in which $w$ and $y$ differ, which must be a subset $\{i,j_1,j_2\}$, cannot be both $j_1$ and $j_2$.  Thus, one of them must be $i$, meaning that $w(i) \neq y(i)$.
\end{proof}

We now focus on a particular kind of interval in the symmetric group, and look at what such an interval implies for the reduced words of its minimum and maximum elements.

\begin{defn}
Let ${\sf s}$ be a string of integers, and $t \in \mathbb{Z}$.  The \emph{shift} of ${\sf s}$ by $t$ is the string ${\sf s}^t$ obtained by adding $t$ to each of the values in ${\sf s}$.
\end{defn}

\begin{ex}
$({\sf 5 \ -\!1 \ 0})^4 = {\sf 9 \ 3 \ 4}$ and $({\sf 5 \ -\!1 \ 0})^{-4} = {\sf 1 \ -\!5 \ -\!4}$.
\end{ex}

\begin{prop}\label{prop:symbols in the factor b}
Fix a positive integer $k$.  Suppose that $x,y \in \mf{S}_n$ have reduced words ${\sf ac} \in R(x)$ and ${\sf abc} \in R(y)$, and that $[x,y] \cong \mf{S}_k$.  Then there exists an integer $t \ge 0$ such that ${\sf b}^{-t} \in R(w_0^k)$.
\end{prop}

\begin{proof}
The length of ${\sf b}$ is equal to $\ell(y) - \ell(x) = \ell(w_0^k) = \binom{k}{2}$.  Because ${\sf b}$ is a reduced word, it cannot contain fewer than $k-1$ distinct symbols.  Moreover, if it were to contain more than $k-1$ distinct symbols, then $[x,y]$ would not be isomorphic to $\poi{w_0^k}$.  Thus ${\sf b}$ contains exactly $k-1$ distinct symbols.  In order to form a reduced word of length $\ell(w_0^k)$ out of $k-1$ distinct symbols, that reduced word must actually be the shift by $t$ of a reduced word for $w_0^k$, where $t+1$ be the smallest symbol in ${\sf b}$.
\end{proof}

\begin{ex}
In the language of Proposition~\ref{prop:symbols in the factor b}, let $x = 21534$ and $y = 24531$, with reduced words ${\sf 143} \in R(x)$ and ${\sf 123243} \in R(y)$. The interval $[x,y] \subset \mf{S}_5$ is isomorphic to $S_3$, as drawn in Figure~\ref{fig:s3 interval}. In this example, $t = 1$ because ${\sf 232}^{-1} \in R(321)$.
\end{ex}

\begin{figure}[htbp]
\begin{tikzpicture}
\draw (0,0) coordinate (21534);
\draw (-1.5,1.5) coordinate (23514);
\draw (1.5,1.5) coordinate (21543);
\draw (-1.5,3) coordinate (23541);
\draw (1.5,3) coordinate (24513);
\draw (0,4.5) coordinate (24531);
\foreach \x in {21534, 23514, 21543, 23541, 24513, 24531} {\fill[black] (\x) circle (2pt);}
\foreach \x in {23514, 23541} {\draw (\x) node[left] {$\x$};}
\foreach \x in {21543, 24513} {\draw (\x) node[right] {$\x$};}
\draw (21534) node[below] {$21534$};
\draw (24531) node[above] {$24531$};
\draw (21534) -- (23514) -- (23541) -- (24531) -- (24513) -- (21543) -- (21534);
\draw (23514) -- (24513);
\draw (21543) -- (23541);
\end{tikzpicture}
\caption{The interval $[21534, 24531] \subset \mf{S}_5$, which is isomorphic to $\mf{S}_3$.}\label{fig:s3 interval}
\end{figure}
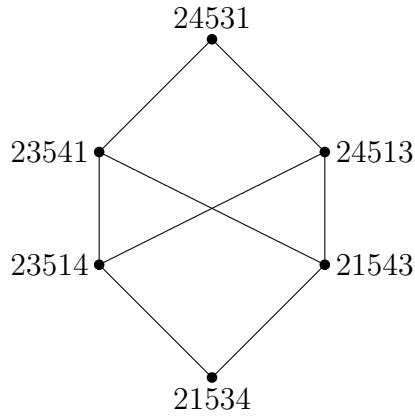

The next two propositions require an additional definition.

\begin{defn}
Fix a string ${\sf s}$.  Consider a monotonic substring ${\sf s'}$ of ${\sf s}$ with smallest value $a$ and largest value $b$ (the endpoints of the monotonic substring).  This ${\sf s'}$ is \emph{thin} if no value $c \not\in {\sf s'}$, with $a < c < b$, appears between $a$ and $b$ in ${\sf s}$.
\end{defn}

\begin{ex}
Let ${\sf s} = {\sf 9 1 4 0 2 3 6 5}$.  The monotonic substring ${\sf 0 2 3 5}$ is thin, while the monotonic substring ${\sf 9 1 0}$ is not thin because of the ${\sf 4}$ appearing between ${\sf 9}$ and ${\sf 0}$ in ${\sf s}$.
\end{ex}

\begin{prop}\label{prop:swap-string}
Fix a positive integer $k$.  Suppose that $x,y \in \mf{S}_n$ have reduced words ${\sf ac} \in R(x)$ and ${\sf abc} \in R(y)$, with ${\sf b}^{-t} \in R(w_0^k)$ for some integer $t \ge 0$.  Then $x$ and $y$, as strings, are identical outside of a thin monotonic substring of length $k$, which appears in increasing order in $x$ and decreasing order in $y$.
\end{prop}

To ease the discussion, we will call this monotonic substring that distinguishes $x$ from $y$ the \emph{swap-string}.  Note that if $k$ is odd, then $x$ and $y$ will also be identical in the central position of the swap-string.

\begin{proof}
We prove the result by induction on the length of ${\sf a}$.

If ${\sf a} = \emptyset$, then
$$y = \big(12\cdots t(t+k)\cdots(t+2)(t+1)\big)x.$$
Thus $x$ and $y$ only differ, as strings, in the subsequence involving $\{t+1,\ldots,t+k\}$, which necessarily appears in increasing order in $x$ (because ${\sf bc}$ is reduced, so multiplying ${\sf c}$ by the permutation corresponding to ${\sf b}$ cannot undo any inversions) and decreasing order in $y$.  Because there are no values between $t+1$ and $t+k$ that are not already in the swap-string, the result holds.

Now suppose that ${\sf a} = {\sf u a'}$ where ${\sf u}$ is a single letter.  Let $x' = s_ux$ and $y' = s_uy$, and assume inductively that the result holds for ${\sf a'c} \in R(x')$ and ${\sf a'bc} \in R(y')$.  This means that the strings $x'$ and $y'$ are identical except for a substring of length $k$ whose values appear in increasing order in $x'$ and decreasing order in $y'$, and these swap-strings are thin.

Because ${\sf ac}$ and ${\sf abc}$ are both reduced words, the value $u$ must appear to the left of $u+1$ in both $x'$ and $y'$.  Thus at most one of $\{u,u+1\}$ appears in the swap-string for $x'$ and $y'$, so swapping the positions of the two values cannot change the monotonicity of the differentiating substrings.  In other words, lengthening the prefix might change the specific values that differ in the two strings, but cannot alter the swap-string phenomenon.

Similarly, the thinness of the swap-string is maintained by this operation.
\end{proof}

The converse to Proposition~\ref{prop:swap-string} is also true.  Its proof is similar to the main result of \cite{rdpp}, and we offer the broad strokes of it here.

\begin{prop}\label{prop:swap-string means factor}
Suppose that $x,y \in \mf{S}_n$ are identical, as strings, outside of a thin monotonic substring of length $k$, which appears in increasing order in $x$ and decreasing order in $y$.  Then there exist reduced words ${\sf ac} \in R(x)$ and ${\sf abc} \in R(y)$, with ${\sf b}^{-t} \in R(w_0^k)$ for some integer $t \ge 0$.
\end{prop}

\begin{proof}
We will transform both $x$ and $y$ into the identity permutation, by minimally many simple reflections, thus obtaining reduced words for each.

Look for any values sitting in between the endpoints of the swap-string that are not actually in the swap-string themselves.  Because the swap-strings are thin, each of these values is either smaller than the minimum value of the swap-string or larger than the maximum value of the swap-string.  Identically multiply $x$ and $y$ on the right by a succession of simple reflections (thus swapping the values in adjacent positions in the strings) to move all of the too-large (respectively, too-small) values to the right (respectively, left) of the swap-string.  The order in which these values are moved can be chosen so that each multiplication removes exactly one inversion.  Let ${\sf C}$ be the reduced word corresponding to the product of these multiplied simple reflections.

We now have two permutations $x'\le x$ and $y'\le y$ that are identical, as strings, outside of a swap-string of length $k$, which appears in increasing order in $x'$ and decreasing order in $y'$.  Moreover, the swap-string is a factor in each of $x'$ and $y'$.  We can now multiply $y'$ on the right by a succession of simple reflections, each of which removes exactly one inversion from the permutation, to put this decreasing factor into increasing order and thus yield $x'$.  This will correspond to a reduced word ${\sf b}$.  Moreover, ${\sf b}$ necessarily satisfies ${\sf b}^{-t} \in R(w_0^k)$, where the leftmost symbol in the swap-string appears in the $(t-1)$st position.  Fix some ${\sf a} \in R(x')$.

Let ${\sf c}$ be the string obtained by writing ${\sf C}$ in reverse order.  Then ${\sf ac} \in R(x)$ and ${\sf abc} \in R(y)$, with ${\sf b}^{-t} \in R(w_0^k)$ for some integer $t \ge 0$, as desired.
\end{proof}

We are now able to prove the main result of this section, describing a family of permutations that force factors.

\begin{thm}\label{thm:forces}
For all integers $n > 1$, the permutation $w_0^n = n(n-1) \cdots 321$ forces a factor.
\end{thm}

\begin{proof}
First note that $\poi{21}\cong\mf{S}_2$ is the following poset.
\begin{center}
\begin{tikzpicture}
\foreach \x in {0,1} {\fill[black] (0,\x) circle (2pt);}
\draw (0,0) -- (0,1);
\end{tikzpicture}
\end{center}
If $[x,y] \cong \mf{S}_2$, then $\ell(y) = \ell(x) + 1$, and so a reduced word for $x$ must be obtained from a reduced word for $y$ by deleting a single letter.  A single letter is necessarily a factor, and so the result holds for $n = 2$.

Now consider some integer $n > 2$, and suppose, inductively, that the result holds for $w_0^{n-1}$.

Let $[x,y] \in \mf{S}_m$ be an interval that is isomorphic to $\mf{S}_n$.  In $\mf{S}_n$, there are two intervals
\begin{equation}\label{eqn:n-1 intervals}
[23\cdots n1,w_0^n] \text{\ \ and\ \ } [n12\cdots (n-1),w_0^n],
\end{equation}
each of which is isomorphic to $\mf{S}_{n-1}$.  Thus, there must be two such intervals $[x_1,y]$ and $[x_2,y]$ in $[x,y]$.  Recall the results of Propositions~\ref{prop:symbols in the factor b} and~\ref{prop:swap-string}.  Let the swap-string for $[x_i,y] \cong \mf{S}_{n-1}$ have values $h_1^{(i)} < \cdots < h_{n-1}^{(i)}$.

Because $[x_1,y]$ and $[x_2,y]$ overlap extensively in $[x,y]$, as do the intervals of \eqref{eqn:n-1 intervals} in $\mf{S}_n$, their respective swap-strings must share many values.  In particular, at the second highest rank in $[x,y]$, the intervals $[x_1,y]$ and $[x_2,y]$ overlap in $n-3$ elements.  Thus the two swap-strings share $n-2$ values.  It remains to determine how these two swap-strings could fit together.  In order to satisfy the thinness condition for the swap-string of $[x_i,y]$, the $n-2$ shared values must be either $\{h_1^{(i)}, \ldots, h_{n-2}^{(i)}\}$ or $\{h_2^{(i)}, \ldots, h_{n-1}^{(i)}\}$.

Note that $[x_1,y] \cup [x_2,y]$ includes all of the coatoms of $[x,y]$.  By Proposition~\ref{prop:all positions mentioned in coatoms}, $x$ and $y$ cannot differ in any positions outside the union of the two  swap-strings.  This union encompasses exactly $n$ positions.  Because the swap-strings are thin, and because $\ell(y) - \ell(x) = \ell(w_0^n)$, we must have the $n$ positions form an increasing substring in $x$ and a decreasing substring in $y$, and these two monotonic substrings must be thin in their respective permutations.  Proposition~\ref{prop:swap-string means factor} now implies that some ${\sf i} \in R(x)$ can be obtained from some ${\sf j} \in R(y)$ by deleting a factor, completing the proof.
\end{proof}

To illustrate the proof of Theorem~\ref{thm:forces}, we present the following example.

\begin{ex}\label{ex:thm}
Let $x = 321456$ and $y = 361542$ in $\mf{S}_6$, for which $[x,y] \cong \mf{S}_4$. In the language of the proof of Theorem~\ref{thm:forces}, then, $n = 4$. Let $x_1 = 341562$ and $x_2 = 361245$. For $i \in \{1,2\}$, the interval $[x_i, y]\subset[x,y]$ is isomorphic to $\mf{S}_3$, as depicted in Figure~\ref{fig:xi intervals}. Note that these intervals share $4 - 3 = 1$ coatom (the permutation $361452 \in \mf{S}_6$), and the swap-string $\{4,5,6\}$ for $[x_1,y]$ shares $4-2 = 2$ values with the swap-string $\{2,4,5\}$ for $[x_2,y]$.

\begin{figure}[htbp]
\begin{tikzpicture}
\draw (0,0) coordinate (a);
\draw (-1.5,1.5) coordinate (b);
\draw (1.5,1.5) coordinate (c);
\draw (-1.5,3) coordinate (d);
\draw (1.5,3) coordinate (e);
\draw (0,4.5) coordinate (f);
\foreach \x in {a, b, c, d, e, f} {\fill[black] (\x) circle (2pt);}
\draw (a) node[below] {$x_1 = 341562$};
\draw (b) node[left] {$351462$};
\draw (c) node[right] {$341652$};
\draw (d) node[left] {$351642$};
\draw (e) node[right] {$361452$};
\draw (f) node[above] {$y = 361542$};
\draw (a) -- (b) -- (d) -- (f) -- (e) -- (c) -- (a);
\draw (b) -- (e);
\draw (c) -- (d);
\end{tikzpicture}
\hspace{.5in}
\begin{tikzpicture}
\draw (0,0) coordinate (a);
\draw (-1.5,1.5) coordinate (b);
\draw (1.5,1.5) coordinate (c);
\draw (-1.5,3) coordinate (d);
\draw (1.5,3) coordinate (e);
\draw (0,4.5) coordinate (f);
\foreach \x in {a, b, c, d, e, f} {\fill[black] (\x) circle (2pt);}
\draw (a) node[below] {$x_2 = 361245$};
\draw (b) node[left] {$361425$};
\draw (c) node[right] {$361254$};
\draw (d) node[left] {$361452$};
\draw (e) node[right] {$361524$};
\draw (f) node[above] {$y = 361542$};
\draw (a) -- (b) -- (d) -- (f) -- (e) -- (c) -- (a);
\draw (b) -- (e);
\draw (c) -- (d);
\end{tikzpicture}
\caption{The intervals $[x_1,y]$ and $[x_2,y]$ described in Example~\ref{ex:thm}, illustrating the proof of Theorem~\ref{thm:forces}.}\label{fig:xi intervals}
\end{figure}
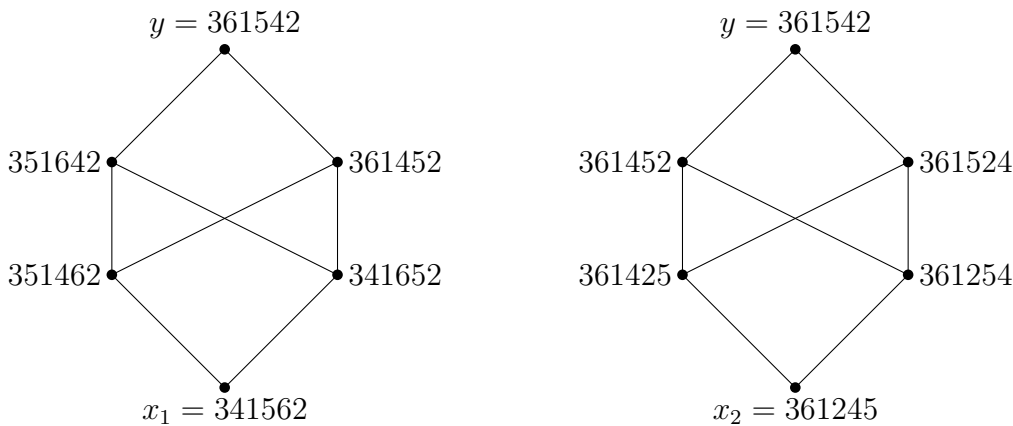
\end{ex}

In fact, Proposition~\ref{prop:swap-string means factor} also tells us the form of the factor forced by $w_0^n$.

\begin{cor}
For all integers $n > 1$, if $[x,y] \cong \mf{S}_n$, then some ${\sf i} \in R(x)$ can be obtained from some ${\sf j} \in R(y)$ by deleting a factor ${\sf b}$, where ${\sf b}^t \in R(w_0^n)$ for some $t \in \mathbb{Z}$.
\end{cor}

\section{Open questions}\label{section:questions}

We have now documented a family of permutations that do not force factors and a second family of permutations that do force factors.  Completely characterizing those permutations that do (or do not) force factors is still an open question, and one which could shed significant light on the interval structure of the Bruhat order for the symmetric group.

In a different direction, the present work studies only the finite Coxeter group of type $A$, although the analogous question can be asked for Coxeter groups of other types as well.

\section*{Acknowledgements}

I am grateful for the thoughtful suggestions of an anonymous referee.

\end{document}